\documentclass[12pt]{article}

\usepackage{amsmath,amssymb,amsthm,multirow,graphicx,ifpdf,mathtools}
\newtheorem{theorem}{Theorem}[section]
\newtheorem{proposition}[theorem]{Proposition}
\newtheorem{lemma}[theorem]{Lemma}

\theoremstyle{definition}
\newtheorem{definition}[theorem]{Definition}
\theoremstyle{remark}
\newtheorem{remark}[theorem]{Remark}
\newtheorem{example}[theorem]{Example}

\numberwithin{equation}{section}

\usepackage{enumitem}

\usepackage{authblk}
\title{A Hilbert space approach to\\ fractional difference equations}
\author[1]{Pham The Anh}
\author[2]{Artur Babiarz}
\author[2]{Adam Czornik}
\author[4]{Konrad Kitzing}
\author[2,3]{Micha\l{} Niezabitowski}
\author[4]{Stefan Siegmund}
\author[4]{Sascha Trostorff}
\author[5]{Hoang The Tuan}
\affil[1]{Department of Mathematics, Le Quy Don Technical University, 236 Hoang Quoc Viet, Ha noi, Vietnam} 
\affil[2]{Silesian University of Technology, Faculty of Automatic Control, Electronics and Computer Science, Akademicka 16, 44-100 Gliwice, Poland}
\affil[3]{University of Silesia, Faculty of Mathematics, Physics and Chemistry, Institute of Mathematics, Bankowa 14, 40-007 Katowice, Poland} 
\affil[4]{Technische Universit\"{a}t Dresden, Faculty of Mathematics, Zellescher Weg 12-14, 01069 Dresden, Germany}
\affil[5]{Institute of Mathematics, Vietnam Academy of Science and Technology, 18 Hoang Quoc Viet Road, Cau Giay, Hanoi, Vietnam} 

\allowdisplaybreaks

\usepackage[mathscr]{eucal}
\usepackage{calrsfs}

\usepackage[utf8]{inputenc}
\usepackage[T1]{fontenc}

\renewcommand{\phi}{\varphi}
\renewcommand{\rho}{\varrho}
\newcommand{\eps}{\varepsilon}
\newcommand{\R}{\mathbb{R}}
\newcommand{\C}{\mathbb{C}}

\newcommand{\N}{\mathbb{N}}

\newcommand{\Z}{\mathbb{Z}}

\newcommand{\cZ}{\mathcal{Z}}

\newcommand{\mulop}{\operatorname{m}}
\newcommand{\spt}{\operatorname{spt}}
\newcommand{\ran}{\operatorname{ran}}
\newcommand{\norm}[1]{\left\lVert#1\right\rVert}
\newcommand{\abs}[1]{\left\lvert#1\right\rvert}

\newcommand{\setm}[2]{\{\hspace*{0,07em}#1;#2\hspace*{0,07em}\}}

\begin{document}

\maketitle              

\begin{abstract}
We formulate fractional difference equations of Riemann-Liouville and Caputo type in a functional analytical framework. Main results are existence of solutions on Hilbert space-valued weighted sequence spaces and a condition for stability of linear fractional difference equations. Using a functional calculus, we relate the fractional sum to fractional powers of the operator $1 - \tau^{-1}$ with the right shift $\tau^{-1}$ on weighted sequence spaces. Causality of the solution operator plays a crucial role for the description of initial value problems.
\end{abstract}

\section{Introduction}

\subsection{Notation}

We write $\R_{> 0} \coloneqq \setm{x \in \R}{x > 0}$ and for $\mu, \rho \in \R$ we define for the comprehension $\C_{\abs{\cdot} < \mu} \coloneqq \setm{z \in \C}{\abs{z} < \mu}$
    and $\C_{\abs{\cdot} > \mu}$, $\C_{\abs{\cdot} \leq \mu}$, $\C_{\abs{\cdot} \geq \mu}$ and $\C_{\mu \geq \abs{\cdot} \geq \rho}$ are defined similarly.
For $\rho > 0$ we denote the complex ball with radius $\rho$ centered at $0$ by $B(0, \rho) \coloneqq \setm{z \in \C}{|z| < \rho}$ and the circle with radius $\rho$ centered at $0$ by $S_\rho \coloneqq \partial B(0, \rho)$.
We set $\N \coloneqq \Z_{\geq 0} \coloneqq \{0, 1, 2, \dots\}$.
For sets $X, Y$ we denote the set of functions from $Y$ to $X$ by $X^Y \coloneqq \{f: Y \rightarrow X\}$ and for $f \in X^Y$ we write $\ran f \coloneqq \setm{f(y) \in X}{y \in Y}$ for the range of $f$.
In particular, for any $M \subseteq \Z$, $X^M$ is the space of sequences in $X$ on $M$ and for $u \in X^M$, $n \in M$ we write $u_n \coloneqq u(n)$.
The identity mapping on a vector space $V$ is denoted by $1$.
For a sequence $u \in V^\Z$ we denote $\spt u \coloneqq \setm{n \in \Z}{u_n \neq 0}$.
If $V$ is a normed vector space we denote with $\norm{\cdot}_V$ the norm on $V$.

We recall the binomial coefficient and the binomial series including some of their properties.
Proofs of the following propositions can be found in \cite{knuth1994} and \cite{koenigsberger2004}.
\begin{proposition}[Binomial coefficient ({\cite[pp.~164-165]{knuth1994}}, {\cite[p.~34]{koenigsberger2004}})]\label{p:binom}
For $\alpha \in \C$ and $n \in \Z_{\geq 1}$ the binomial coefficent is defined by
\begin{equation*}
    \binom{\alpha}{0} \coloneqq 1, \qquad \binom{\alpha}{n} \coloneqq \frac{\alpha(\alpha - 1) \cdots (\alpha - n + 1)}{n!}.
\end{equation*}
For $\alpha \in \C$ and $n \in \N$ we have
\begin{equation*}
    (-1)^n \binom{\alpha}{n} = \binom{-\alpha + n - 1}{n} \qquad \text{and} \qquad \sum_{k = 0}^n (-1)^k \binom{\alpha}{k} = (-1)^n \binom{\alpha - 1}{n}.
\end{equation*}
\end{proposition}

\begin{proposition}[Binomial series ({\cite[p.~65 \& p.~73]{koenigsberger2004}})]
Let $\alpha \in \C$.
The binomial power series is defined by
\begin{equation*}
    (1 + z)^\alpha \coloneqq \sum_{k = 0}^\infty \binom{\alpha}{k} z^k.
\end{equation*}
The series converges absolutely in $B(0, 1)$.
In particular, the mapping $\C_{\abs{\cdot} > 1} \rightarrow \C, z \mapsto (1 - z^{-1})^\alpha$ is holomorphic.
For each $\alpha, \beta \in \C$ we have $(1 + z)^\alpha (1 + z)^\beta = (1 + z)^{\alpha + \beta}$.
\end{proposition}

Binomial coefficients can be expressed with the gamma function.

\begin{lemma}[Falling factorial ({\cite[p.~164]{knuth1994}})]\label{l:eq-binom-gamma}
With the falling factorial
\begin{equation*}
    (x)^{(n)} \coloneqq \frac{\Gamma(x + 1)}{\Gamma(x - n + 1)}
    , \qquad x \in \C \setminus \Z, \qquad n \in \N,
\end{equation*}
we have for each $\alpha \in \C \setminus \Z$ and $n \in \N$
\begin{equation}\label{e:factorial-vs-binom}
    (-1)^n \binom{\alpha}{n} = \binom{-\alpha + n - 1}{n} = \frac{1}{\Gamma(-\alpha)} (n - (1 + \alpha))^{(-(1 + \alpha))}.
\end{equation}
\end{lemma}

\begin{lemma}\label{lem:binom-estimate}
Let $\alpha \in (0, 1)$ and $\rho > 1$.
Then we have for each $z \in S_\rho$
\begin{equation*}
    (1 - \rho^{-1})^\alpha \leq \abs{(1 - z^{-1})^\alpha}.
\end{equation*}
\end{lemma}

\begin{proof}
Let $z \in S_\rho$.
For every $n \in \Z_{\geq 1}$ we observe that $(-1)^n \binom{\alpha}{n} < 0$ and therefore $(-1)^n \binom{\alpha}{n} z^{-n} = -\abs{\binom{\alpha}{n}} z^{-n}$.
We show by induction that for every $n \in \N$
\begin{equation*}
    \abs{\sum_{k = 0}^n (-1)^k \binom{\alpha}{k} z^{-k}} \geq \abs{\sum_{k = 0}^n (-1)^k \binom{\alpha}{k} \rho^{-k}}
\end{equation*}
and when letting $n$ tend to infinity the inequality follows.
The induction basis is trivial.
For the induction step for $n \in \N$ we use the lower triangle inequality to obtain
\begin{align*}
    \abs{\sum_{k = 0}^{n + 1} (-1)^k \binom{\alpha}{k} z^{-k}} &= \abs{\sum_{k = 0}^n (-1)^k \binom{\alpha}{k} z^{-k} + (-1)^{n + 1} \binom{\alpha}{n + 1} z^{-(n + 1)}}
    \\
    &\geq \abs{\abs{\sum_{k = 0}^n (-1)^k \binom{\alpha}{k} z^{-k}} - \abs{(-1)^{n + 1} \binom{\alpha}{n + 1} z^{-(n + 1)}}}
    \\
    &= \abs{\abs{\sum_{k = 0}^n (-1)^k \binom{\alpha}{k} z^{-k}} + (-1)^{n + 1} \binom{\alpha}{n + 1} \rho^{-(n + 1)}}
    \\
    &\geq \abs{\sum_{k = 0}^{n + 1} (-1)^k \binom{\alpha}{k} \rho^{-k}}.  \qedhere
\end{align*}
\end{proof}

\subsection{Fractional difference operators}

Let $V$ be a real or complex vector space.

The fractional sum can be motivated by the iterated sum formula and is also related to iterating the backward difference operator (see e.g.\ \cite{kuttner1957}). 
For $\alpha \in \R_{> 0}$ the fractional sum $\nabla^{-\alpha} \colon V^\N \to V^\N$ is defined by (cf.\ \cite[p.\ 3]{atici2009})
\begin{equation}\label{e:FI-op}
   (\nabla^{-\alpha} u)_n 
   = 
   \sum_{k = 0}^n \binom{n - k + \alpha - 1}{n - k} u_k 
   = 
   \sum_{k = 0}^n (-1)^k \binom{-\alpha}{k} u_{n - k}.
\end{equation}
There is also a definition motivated by iterating the forward difference operator which is studied at least since \cite{kuttner1957} and can be found in \cite[p.\ 3]{atici2009} as well.
Note that $(\nabla^{-\alpha} u)_n$ in general depends on $u_0, \dots, u_n$.

The approach to defining the fractional differential operators in the Riemann-Liouville and Caputo sense (cf.\ \cite{Diethelm2004}) was applied mutatis mutandis to difference operators (see e.g.\ \cite{Qasem2013} and the references therein).
Recall that for $\Delta: V^\N \rightarrow V^\N, u \mapsto (u_{n + 1} - u_n)_\N$ we have $(\Delta u)_n = (\nabla u)_{n + 1}$ for $n \in \N$. For $\alpha \in (0, 1)$
    the Riemann-Liouville forward fractional difference operator is defined by (cf.\ \cite[p.\ 3813]{Lizama2017})
\begin{equation}\label{e:RL-op}
    \Delta^{\alpha}: V^\N \rightarrow V^\N, \qquad u \mapsto \Delta \nabla^{-(1 - \alpha)} u.
\end{equation}
The Caputo forward fractional difference operator is defined by (cf. \cite[p. 3813]{Lizama2017})
\begin{equation}\label{e:C-op}
    \Delta_C^\alpha: V^\N \rightarrow V^\N, \qquad u \mapsto \nabla^{-(1 - \alpha)} \Delta u.
\end{equation}
In this paper we study sequences in a Hilbert space $V = H$ on $\Z$
	and define a fractional difference sum operator using the binomial series and a functional calculus
	which is not purely algebraic as in the case of $\nabla^{-\alpha}$.
The connection between operators defined on $H^\Z$ with those defined on $H^\N$ will be causality
    and we analyze how the Riemann-Liouville and the Caputo operator fit into the calculus developed for sequences in $H^\Z$.
An important step for the development of the discrete, functional analytic framework which is introduced in this paper has been done in the continuous case for fractional derivatives in \cite{Picard:Trostorff:Waurick2015}.
Lastly we study the asymptotic stability of the zero solution of a linear fractional difference equation with the Riemann-Liouville and the Caputo forward difference operator.
The interest in the study of linear problems in the context of stability analysis stems from Lyapunov's first method, which has been analyzed in \cite{Siegmund2016} for fractional differential equations.
The results regarding asymptotic stability will be in terms of the Matignon criterion (cf.\ \cite{Matignon1998}), however, for bounded operators on a Hilbert space $H$
	and will be compared to those in \cite{Qasem2013} and \cite{czermak2015}.
A useful tool when analyzing the asymptotic stability of linear problems is the $\cZ$ transform which is also used in \cite{Qasem2013} and \cite{czermak2015}
    but which is studied here for sequences in $H^\Z$.
Asymptotic stability has also been studied using the Riemann-Liouville and the Caputo backward difference operators in \cite{czermak2012} and \cite{Lizama2017}.

\section{Exponentially weighted $\ell_p$ spaces}

We denote by $(H, \norm{\cdot}_H)$ a complex and separable Hilbert space.
The scalar product $\left<\cdot, \cdot\right>_H$ on $H$ shall be conjugate linear in the first argument and linear in the second argument.
We recall several of the concepts of weighted $\ell_{p, \rho}(\Z; H)$ spaces and the $\cZ$ transform (see also \cite{Siegmund2018}).

\begin{lemma}[Exponentially weighted $\ell_p$ spaces \cite{Siegmund2018}]\label{lem:ell2rho}
Let  $1 \leq p < \infty$, $\rho >0$. Define
\begin{align*}
   \ell_{p,\rho}(\Z; H)
   &\coloneqq
   \setm{x \in H^\Z}{\sum_{k \in \Z} \norm{x_k}_H^p \rho^{-p k} < \infty},
\\
   \ell_{\infty,\rho}(\Z; H)
   &\coloneqq
   \setm{x \in H^\Z}{\sup_{k \in \Z} \norm{x_k}_H \rho^{-k} < \infty}.
\end{align*}
Then $\ell_{p,\rho}(\Z; H)$ and $\ell_{\infty,\rho}(\Z; H)$ are Banach spaces with norms
\begin{equation*}
   \norm{x}_{\ell_{p,\rho}(\Z; H)} 
   \coloneqq
   \Big( \sum_{k \in \Z} \norm{x_k}_H^p \rho^{-p k} \Big)^{\frac{1}{p}}
   \qquad
   (x \in \ell_{p,\rho}(\Z; H))
\end{equation*}
and
\begin{equation*}
   \norm{x}_{\ell_{\infty,\rho}(\Z; H)}
   \coloneqq
   \sup_{k \in \Z} \norm{x_k}_H \rho^{-k}
   \qquad
   (x \in \ell_{\infty,\rho}(\Z; H)),
\end{equation*}
respectively. Moreover, $\ell_{2,\rho}(\Z; H)$ is a Hilbert space with the inner product
\begin{equation*}
   \langle x, y \rangle_{\ell_{2,\rho}(\Z; H)}
   \coloneqq
   \sum_{k \in \Z} \big\langle x_k, y_k \big\rangle_{\!H} \rho^{-2 k}
   \qquad
   (x, y \in \ell_{2,\rho}(\Z; H)).
\end{equation*}
We write $\ell_{p}(\Z; H) \coloneqq \ell_{p,1}(\Z; H)$ for $1 \leq p \leq \infty$.
\end{lemma}

\begin{proposition}[One sided weighted sequence spaces \cite{Siegmund2018}]\label{lem:one-sided}
For $1 \leq p \leq \infty$, $a \in \Z$ and $\rho > 0$ we define
\begin{equation*}
   \ell_{p,\rho}(\Z_{\geq a}; H)
   \coloneqq
   \setm{x|_{\Z_{\geq a}}}{x \in \ell_{p,\rho}(\Z; H)}.
\end{equation*}
And for $1 \leq p \leq \infty$, $\rho > 0$, $a \in \Z$
	and for $x \in H^{\Z_{\geq a}}$, we define $\iota x \in H^{\Z}$ by
\begin{equation*}
   (\iota x)_k
   \coloneqq
   \begin{cases}
      0 & \text{if } k < a,
   \\
      x_k & \text{if } k \geq a.
   \end{cases}
\end{equation*}
Then $\ell_{p,\rho}(\Z_{\geq a}; H)$ is a Banach space with norm $\norm{\cdot}_{\ell_{p,\rho}(\Z_{\geq a}; H)} \coloneqq  \norm{\iota \cdot}_{\ell_{p,\rho}(\Z; H)}$, and
\begin{equation*}
   \iota \colon \ell_{p,\rho}(\Z_{\geq a}; H) \hookrightarrow \ell_{p,\rho}(\Z; H)
\end{equation*}
is an isometric embedding. Write $\ell_{p,\rho}(\Z_{\geq a}; H) \subseteq \ell_{p,\rho}(\Z; H)$.
\\[1ex]
For $1 \leq p < q \leq \infty$, $\rho, \eps > 0$, $a \in \Z$ we have
\begin{align*}
	(a) \qquad &\ell_{p,\rho}(\Z_{\geq a}; H) \subsetneq \ell_{q,\rho}(\Z_{\geq a}; H),
	\\
	(b) \qquad &\ell_{q,\rho}(\Z_{\geq a}; H) \subsetneq \ell_{p,\rho + \eps}(\Z_{\geq a}; H).
\end{align*}
\end{proposition}

\begin{definition}
For $x \in H$ and $n \in \Z$ we define $\delta_n x \in H^\Z$ by
\begin{equation*}
    (\delta_n x)_m \coloneqq
    \begin{cases}
    x, &\text{if} \; m = n,
    \\
    0, &\text{if} \; m \neq n,
    \end{cases}
\end{equation*}
and $\chi_{\Z_{\geq n}} x \in H^\Z$ by
\begin{equation*}
    (\chi_{\Z_{\geq n}} x)_m \coloneqq
    \begin{cases}
    x, &\text{if} \; m \geq n,
    \\
    0, &\text{if} \; m < n.
    \end{cases}
\end{equation*}
\end{definition}
Note that for $\rho > 0$, $\delta_n x \in \ell_{p, \rho}(\Z; H)$
    and for $\rho > 1$, $\chi_{\Z_{\geq n}} x \in \ell_{p, \rho}(\Z; H)$.

\begin{lemma}[Shift operator \cite{Siegmund2018}]\label{lem:Shift}
Let $1 \leq p \leq \infty$, $\rho > 0$. Then
\begin{align*}
   \tau \colon \ell_{p,\rho}(\Z; H) & \rightarrow \ell_{p,\rho}(\Z; H),
\\
   (x_k)_{k \in \Z} & \mapsto (x_{k+1})_{k \in \Z},
\end{align*}
is linear, bounded, invertible, and 
\begin{equation*}
   \norm{\tau^n}_{L(\ell_{p,\rho}(\Z; H))} = \rho^n
   \qquad
   (n \in \Z).
\end{equation*}
\end{lemma}

\section{$\cZ$ transform}

\begin{lemma}[$L_2$ space on circle and orthonormal basis \cite{Siegmund2018}]
Let $\rho >0$. Define
\begin{equation*}
   L_{2}(S_\rho; H)
   \coloneqq
   \setm{f \colon S_\rho \rightarrow H}{ \int_{S_\rho} \norm{ f(z) }_H^2 \frac{\mathrm{d}z}{|z|} < \infty}.
\end{equation*}
Then $L_{2}(S_\rho; H)$ is a Hilbert space with the inner product
\begin{equation*}
   \langle f, g \rangle_{L_{2}(S_\rho; H)}
   \coloneqq
   \frac{1}{2\pi}
   \int_{S_\rho} \big\langle f(z), g(z) \big\rangle_{\!H} \frac{\mathrm{d}z}{|z|}
   \qquad
   (f, g \in L_{2}(S_\rho; H)).
\end{equation*}
Moreover, let $(\psi_n)_{n \in \Z}$ be an orthonormal basis in $H$.
Then $(p_{k,n})_{k, n \in \Z}$ with
\begin{equation*}
   p_{k,n}(z) \coloneqq \rho^{k} z^{-k} \psi_n
   \qquad (z \in S_\rho)
\end{equation*}
is an orthonormal basis in $L_2(S_\rho; H)$.
\end{lemma}

\begin{theorem}[$\cZ$ transform \cite{Siegmund2018}]\label{t:ztransform}
Let $\rho >0$.
The operator
\begin{align*}
   \cZ_\rho \colon \ell_{2,\rho}(\Z; H) &\rightarrow L_2(S_\rho; H),
\\
   x & \mapsto \Big( z \mapsto \sum_{k \in \Z} \langle \psi_n, \rho^{-k} x_k\rangle_{\!H} p_{k, n}(z) \Big)
\end{align*}
is well-defined and unitary.
For $x \in \ell_{1, \rho}(\Z; H) \subseteq \ell_{2, \rho}(\Z; H)$ we have
\begin{equation*}
    \cZ_\rho(x) = \left(z \mapsto \sum_{k \in \Z} x_k z^{-k}\right).
\end{equation*}
\end{theorem}

\begin{remark}[$\cZ$ transform of $x \in \ell_{2,\rho}(\Z; H) \setminus \ell_{1,\rho}(\Z; H)$]\label{rem:z-transform}
Let $\rho > 0$, $x \in \ell_{2,\rho}(\Z; H) \setminus \ell_{1,\rho}(\Z; H)$. Then
\begin{equation*}
   \sum_{k \in \Z} x_k z^{-k} 
\end{equation*}
does not necessarily converge for all $z \in S_\rho$.
For example if $H = \C$, $x \in \ell_{2, \rho}(\Z; H) \setminus \ell_{1, \rho}(\Z; H)$ with $x_k \coloneqq \frac{\rho^k}{k}$ and $z = \rho$.
\end{remark}

\begin{lemma}[Shift is unitarily equivalent to multiplication \cite{Siegmund2018}]\label{lem:mo} Let $\rho > 0$. Then 
\[
   \mathcal{Z}_\rho \tau \mathcal{Z}_\rho^* = \mathrm{m},
\]
where $\mathrm{m}$ is the multiplication-by-the-argument operator acting in $L_2(S_\rho;H)$, i.e.,
\begin{align*}
   \mathrm{m}\colon L_2(S_\rho;H) &\to L_2(S_\rho;H),
\\
   f&\mapsto (z\mapsto zf(z)).
\end{align*} 
\end{lemma}

Next, we present a Paley--Wiener type result for the $\mathcal{Z}$ transform. 

\begin{lemma}[Characterization of positive support \cite{Siegmund2018}]\label{lem:positive-support}
Let $\rho > 0$, $x \in \ell_{2,\rho}(\Z; H)$. Then the following statements are equivalent:

(i) $\spt x \subseteq \N$,

(ii) $z \mapsto \sum_{k \in \Z} x_k z^{-k}$ is analytic on $\C_{\abs{\cdot} > \rho}$ and
\begin{equation}\label{e:Hardybound}
   \sup_{\mu > \rho} \int_{S_\mu} \norm{\sum_{k \in \Z} x_k z^{-k}}_H^2 \,\frac{\mathrm{d}z}{|z|} < \infty.
\end{equation}
\end{lemma}

\begin{definition}[Causal linear operator]\label{d:linearcausal}
    We call a linear operator $B  \colon \ell_{2,\rho}(\Z; H) \rightarrow \ell_{2,\rho}(\Z; H)$ \emph{causal}, if for all $a\in \Z$, $f\in \ell_{2,\rho}(\Z;H)$, we have
    \begin{equation*}
        \spt f \subseteq \mathbb{Z}_{\geq a} 
        \quad\Rightarrow\quad
        \spt Bf \subseteq\mathbb{Z}_{\geq a}.
    \end{equation*}
\end{definition}

Recall \cite[VIII.3.6, p.\ 222]{Katznelson2004} that for $A \in L(H)$ with spectrum $\sigma(A)$, the spectral radius
\begin{equation*}
   r(A)
   \coloneqq
   \sup \setm{|z|}{z \in \sigma(A)}
\end{equation*}
of $A$ satisfies
\begin{equation*}
   r(A)
   =
   \lim_{n \to \infty} \norm{A^n}_{L(H)}^{1/n}.
\end{equation*}
Let $A\in L(H)$ and $\rho > 0$.
We denote the operators $\ell_{2,\rho}(\Z, H)\rightarrow \ell_{2,\rho}(\Z, H)$, $x\mapsto (Ax_k)$,
        and $L_2(S_\rho, H)\rightarrow L_2(S_\rho, H)$, $f\mapsto (z \mapsto Af(z))$,
        which have the same operator norm as $A$, again by $A$.

\begin{proposition}[Convolution]
Let $c \in \ell_{1, \rho}(\Z; \C)$ and $u \in \ell_{2, \rho}(\Z; H)$.
Then
\begin{equation*}
	c * u \coloneqq \left(\sum_{k = -\infty}^\infty c_k u_{n - k}\right)_{n \in \Z} \in \ell_{2, \rho}(\Z; H).
\end{equation*}
We have Young's inequality
\begin{equation*}
	\norm{c * u}_{\ell_{2, \rho}(\Z; H)} \leq \norm{c}_{\ell_{1, \rho}(\Z; \C)} \norm{u}_{\ell_{2, \rho}(\Z; H)}.
\end{equation*}
Moreover,
\begin{equation*}
	\cZ_\rho(c * u) = \cZ_\rho c \cZ_\rho u.
\end{equation*}
\end{proposition}

\begin{proof}
Let $n \in \Z$.
With the Cauchy-Schwarz inequality we compute
\begin{align*}
	\left(\sum_{k = -\infty}^\infty \norm{c_k u_{n - k}}_H \right)^2 \rho^{-2n}
		&= \left(\sum_{k = -\infty}^\infty |c_k|^{1/2} \rho^{-k/2} |c_k|^{1/2} \rho^{-k/2} \norm{u_{n - k}}_H \rho^{-(n - k)}\right)^2
	\\
	&\leq \norm{c}_{\ell_{1, \rho}(\Z; \C)} \left(\sum_{k = -\infty}^\infty |c_k| \rho^{-k} \norm{u_{n - k}}_H^2 \rho^{-2(n - k)}\right).
\end{align*}
Therefore using Fubini's theorem
\begin{align*}
	\sum_{n = -\infty}^\infty \norm{(c * u)_n}_H^2 \rho^{-2n}
		&\leq \sum_{n = -\infty}^\infty \left(\sum_{k = -\infty}^\infty \norm{c_k u_{n - k}}_H\right)^2 \rho^{-2n}
	\\
	&\leq \norm{c}_{\ell_{1, \rho}(\Z; \C)}
		\sum_{n = -\infty}^\infty \left(\sum_{k = -\infty}^\infty |c_k| \rho^{-k} \norm{u_{n - k}}_H^2 \rho^{-2(n - k)}\right)
	\\
	&= \norm{c}_{\ell_{1, \rho}(\Z; \C)}^2 \norm{u}_{\ell_{2, \rho}(\Z; H)}^2.
\end{align*}
This shows Young's inequality.
If additionally $u \in \ell_{1, \rho}(\Z; H)$ then
\begin{align*}
	\sum_{n = -\infty}^\infty \norm{(c * u)_n}_H \rho^{-n} &\leq \sum_{n = -\infty}^\infty \sum_{k = -\infty}^\infty \norm{c_k u_{n - k}}_H \rho^{-n}
	\\
	&= \sum_{k = -\infty}^\infty \sum_{n = -\infty}^\infty |c_k| \rho^{-k} \norm{u_{n - k}}_H \rho^{-(n-k)}
	\\
	&= \norm{c}_{\ell_{1, \rho}(\Z; \C)} \norm{u}_{\ell_{1, \rho}(\Z; H)},
\end{align*}
i.e., $c * u \in \ell_{1, \rho}(\Z; H) \cap \ell_{2, \rho}(\Z; H)$ which simplifies the $\cZ$ transform of $c * u$.
Using Fubini's theorem, we compute for $u \in \ell_{1, \rho}(\Z; H) \cap \ell_{2, \rho}(\Z; H)$ and $z \in S_\rho$
\begin{align*}
	\cZ_\rho(c * u)(z) &= \sum_{n = -\infty}^\infty \left(\sum_{k = -\infty}^\infty c_{k} u_{n - k}\right) z^{-n} = \sum_{n = -\infty}^\infty \left(\sum_{k = -\infty}^\infty c_{k} z^{-k} u_{n - k} z^{-(n - k)}\right)
	\\
        &= \sum_{k = -\infty}^\infty c_{k} z^{-k} \left(\sum_{n = -\infty}^\infty  u_{n - k} z^{-(n - k)}\right) = \cZ_\rho(c) \cZ_\rho(u).
\end{align*}
For $u \in \ell_{2, \rho}(\Z; H)$ the formula follows by density of $\ell_{1, \rho}(\Z; H) \cap \ell_{2, \rho}(\Z; H) \subseteq \ell_{2, \rho}(\Z; H)$.
\end{proof}

\begin{example}[The operator $(1 - \tau^{-1})^\alpha$]
Let $\rho > 1$ and $\alpha \in \C$.
For the operator $1 - \tau^{-1}: \ell_{2, \rho}(\Z; H) \rightarrow \ell_{2, \rho}(\Z; H)$, we compute
\begin{equation*}
    (1 - \tau^{-1}) = \cZ_\rho^* (1 - z^{-1}) \cZ_\rho.
\end{equation*}
We have $|z^{-1}| < 1$ for all $z \in S_\rho$ and therefore
\begin{equation*}
    (1 - \tau^{-1})^{\alpha} \coloneqq \cZ_\rho^* (1 - z^{-1})^{\alpha} \cZ_\rho : \ell_{2, \rho}(\Z; H) \rightarrow \ell_{2, \rho}(\Z; H)
\end{equation*}
is well-defined.
This is an application of the holomorphic functional calculus (cf.\ \cite[pp.~13--18]{Gohberg:Goldberg:Kaashoek1990}, \cite[p.~601]{Dunford1988}).

We define $c \in \ell_{1, \rho}(\Z; \C)$ by
\begin{equation*}
    c_k \coloneqq
    \begin{cases}
        (-1)^k \binom{\alpha}{k} & \text{if } k \geq 0,
        \\
        0 & \text{if } k < 0.
    \end{cases}
\end{equation*}
Then
\begin{equation*}
    \cZ_\rho c = \sum_{k = 0}^\infty (-1)^k \binom{\alpha}{k} z^{-k} = (1 - z^{-1})^\alpha.
\end{equation*}
Thus we compute for $u \in \ell_{2, \rho}(\Z; H)$
\begin{equation*}
    \cZ_\rho(c * u) = \cZ_\rho c \cZ_\rho u = (1 - z^{-1})^\alpha \cZ_\rho u.
\end{equation*}
Thus for $\alpha \in \C$ and $u \in \ell_{2, \rho}(\Z; H)$ we obtain
\begin{equation*}
    (1 - \tau^{-1})^\alpha u = c * u = \left(\sum_{k = 0}^\infty (-1)^k \binom{\alpha}{k} u_{n - k}\right)_{n \in \Z}
		= \left(\sum_{k = -\infty}^n (-1)^{n - k} \binom{\alpha}{n - k} u_k\right)_{n \in \Z},
\end{equation*}
i.e., $(1 - \tau^{-1})^\alpha$ is a convolution operator
	and by Young's Theorem $(1 - \tau^{-1})^\alpha$ is bounded and $\norm{(1 - \tau^{-1})^\alpha}_{L(\ell_{2, \rho}(\Z; H))} = \norm{c}_{\ell_{1, \rho}(\Z; H)}$.

If $u \in \ell_{2, \rho}(\Z; H)$ with $\spt u \subseteq \N$, we have
\begin{equation*}
    (1 - \tau^{-1})^{\alpha} u = \left(\sum_{k = 0}^n (-1)^k \binom{\alpha}{k} u_{n - k}\right)_{n \in \Z}.
\end{equation*}
Since $\tau$ commutes with $(1 - \tau^{-1})^{\alpha}$, we deduce that $(1 - \tau^{-1})^\alpha$ is causal.

On $\ell_{2, \rho}(\Z; H)$ we compute for $\alpha, \beta \in \C$
\begin{align*}
    (1 - \tau^{-1})^\alpha (1 - \tau^{-1})^\beta &= \cZ_\rho^* (1 - z^{-1})^\alpha \cZ_\rho \cZ_\rho^* (1 - z^{-1})^\beta \cZ_\rho
	\\
	&= \cZ_\rho^* (1 - z^{-1})^{\alpha + \beta}\cZ_\rho = (1 - \tau^{-1})^{\alpha + \beta}.
\end{align*}
In particular, for $\alpha \in \C$, $(1 - \tau^{-1})^\alpha$ is invertible with inverse $(1 - \tau^{-1})^{-\alpha}$.
\end{example}

\section{Fractional difference equations on $\ell_{2, \rho}(\Z; H)$}

\subsection*{Fractional operators}

Let $\rho > 1$ and $\alpha \in (0, 1)$.
We consider the operators \eqref{e:FI-op}, \eqref{e:RL-op} and \eqref{e:C-op} defined on $V = H$.
For comparing operators defined on spaces of sequences on $\Z$ with those defined for sequences on $\N$, we recall the embedding of $\ell_{2, \rho}(\N; H)$ into $\ell_{2, \rho}(\Z; H)$ by $\iota$ in Lemma \ref{lem:one-sided}.
Moreover, we extend the operator $\Delta$ on $\N$ to $\Z$ by
\begin{equation*}
	\Delta: \ell_{2, \rho}(\Z; H) \rightarrow \ell_{2, \rho}(\Z; H), \qquad u \mapsto \chi_{\N} (\tau - 1)u =
		\chi_{\N} \tau (1 - \tau^{-1})u.
\end{equation*}
Note that the left shift on $\N$ cuts of the first value of a sequence and embedded sequences have positive support.
This is the reason for multiplying with $\chi_{\N}$ in the definition of $\Delta$ on $\ell_{2, \rho}(\Z; H)$.

Let $v \in \ell_{2, \rho}(\N; H)$ and set $u \coloneqq \iota v \in \ell_{2, \rho}(\Z; H)$.
We compare the operator $(1 - \tau^{-1})^{-\alpha}$ defined on $\ell_{2, \rho}(\Z; H)$ and the fractional sum \eqref{e:FI-op}.
We have $\spt \left((1 - \tau^{-1})^{-\alpha} u\right) \subseteq \N$ and obtain
\begin{equation*}
    \iota \nabla^{-\alpha} v = (1 - \tau^{-1})^{-\alpha} u.
\end{equation*}

Using definitions \eqref{e:RL-op} and \eqref{e:C-op} of the Riemann-Liouville and Caputo difference operators, and the fact that $\Delta u = (\tau - 1) (u - \chi_{\N} u_0) = \tau (1 - \tau^{-1}) (u - \chi_{\N} u_0)$, we compute
\begin{align*}
    &\Delta (1 - \tau^{-1})^{-(1 - \alpha)} u = \chi_{\N} \tau (1 - \tau^{-1})^\alpha u
        = \tau (1 - \tau^{-1})^\alpha u - \delta_{-1} u_0,
    \\
    &(1 - \tau^{-1})^{\alpha - 1} \Delta u = (1 - \tau^{-1})^{\alpha - 1} \chi_{\N} \tau (1 - \tau^{-1}) u = \tau (1 - \tau^{-1})^\alpha (u - \chi_{\N} u_0).
\end{align*}
Moreover, we have
\begin{align*}
    &\iota \Delta^\alpha v = \chi_{\N} \tau (1 - \tau^{-1})^{\alpha} u,
    \\
    &\iota \Delta_C^\alpha v = \tau (1 - \tau^{-1})^\alpha (u - \chi_{\N} u_0).
\end{align*}
In view of $\tau (1 - \tau^{-1})^\alpha$, the Caputo and the Riemann-Liouville operators are equal whereby the Caputo operator regularizes $u$ first.
In particular for $n \in \N$ by Proposition \ref{p:binom} we have $((1 - \tau^{-1})^\alpha \chi_\N u_0)_n = \sum_{k = 0}^n (-1)^k \binom{\alpha}{k} u_0 = \binom{-\alpha + n}{n} u_0$ and so
\begin{equation*}
    (\Delta^\alpha v)_n = (\Delta_C^\alpha v)_n + \binom{-\alpha + n + 1}{n + 1} u_0.
\end{equation*}
It is notable that the operator $(1 - \tau^{-1})^\alpha$ defined on $\C$ maps real valued sequences to real valued sequences.
We could have started with a real Hilbert space $H$ and analyze $(1 - \tau^{-1})^\alpha$ spectral-wise by the complexification $H \oplus H$.

\begin{proposition}[Equivalence of difference equation and sequence equation]\label{p:equ-frac-de-seq-eq}
Let $\rho > 1$ and $\alpha \in (0, 1)$.
Let $x \in H$, $F: \ell_{2, \rho}(\Z; H) \rightarrow \ell_{2, \rho}(\Z; H)$ and $u \in \ell_{2, \rho}(\Z; H)$.
Let $\spt u \subseteq \N$ and $\spt F(u) \subseteq \N$.
In view of the Riemann-Liouville operator, the following are equivalent:
\begin{align*}
    &(i) &&\tau (1 - \tau^{-1})^\alpha u = F(u) + \delta_{-1} x,
	\\
    &(ii) &&u_0 = x, ((1 - \tau^{-1})^\alpha u)_{n + 1} = F(u)_n \text{ for } n \in \N,
	\\
    &(iii) &&u_0 = x, u_{n + 1} =
		(-1)^{n + 1} \binom{-\alpha}{n + 1} u_0 + \sum_{k = 0}^n (-1)^{n - k} \binom{-\alpha}{n - k} F(u)_k \text{ for } n \in \N.
\end{align*}
In view of the Caputo operator, the following are equivalent:
\begin{align*}
    &(iv) &&\tau (1 - \tau^{-1})^\alpha u = F(u) + (1 - \tau^{-1})^\alpha \chi_{\Z_{\geq -1}} x,
	\\
    &(v) &&u_0 = x, ((1 - \tau^{-1})^\alpha u)_{n + 1} = F(u)_n + (-1)^{n + 1} \binom{\alpha - 1}{n + 1} u_0 \text{ for } n \in \N,
	\\
    &(vi) &&u_0 = x, u_{n + 1} = u_0 + \sum_{k = 0}^n (-1)^{n - k} \binom{-\alpha}{n - k} F(u)_k \text{ for } n \in \N.
\end{align*}
\end{proposition}

\begin{proof}
We only proof the equivalence of $(i), (ii)$ and $(iii)$.
\\[1ex]
$(i) \Leftrightarrow (ii)$:
If we evaluate $(i)$ at $n \in \Z$ we obtain
\begin{equation*}
    (\tau (1 - \tau^{-1})^\alpha u)_n = ((1 - \tau^{-1})^\alpha u)_{n + 1} = F(u)_n + (\delta_{-1} x)_n.
\end{equation*}
Since $((1 - \tau^{-1})^\alpha u)_n$ and $F(u)_n = 0$ for $n \in \Z_{< 0}$, and since $(\delta_{-1} x)_n = x$ if and only if $n = -1$ and $((1 - \tau^{-1})^\alpha u)_0 = u_0$,
    it follows that $(i)$ and $(ii)$ are equivalent.
\\[1ex]
$(i) \Leftrightarrow (iii)$:
If we apply $(1 - \tau^{-1})^{-\alpha}$ to $(i)$ we see that $(i)$ is equivalent to
\begin{equation*}
	\tau u = (1 - \tau^{-1})^{-\alpha} F(u) + (1 - \tau^{-1})^{-\alpha} \delta_{-1} u.
\end{equation*}
This equation is equivalent to $(iii)$, since
\begin{equation*}
	(1 - \tau^{-1})^{-\alpha} \delta_{-1} x =
	\begin{cases}
		0, &\text{if} \; n < -1,
		\\
		 (-1)^{n + 1} \binom{-\alpha}{n + 1} x, &\text{if} \; n \geq -1,
	\end{cases}
\end{equation*}
and since $\spt F(u) \subseteq \N$,
\begin{equation*}
	(1 - \tau^{-1})^{-\alpha} F(u) 
	= 
	\sum_{k = 0}^n (-1)^{n - k} \binom{-\alpha}{n - k} F(u)_k.
	\qedhere
\end{equation*}
\end{proof}

\begin{remark}
Note that the right hand side $F$ in Proposition \ref{p:equ-frac-de-seq-eq}$(i), (iv)$ maps sequences instead of values of $H$.
If we have a function $f: H \rightarrow H$ such that for $u \in \ell_{2, \rho}(\Z; H)$ we have $(f(u_n))_{n \in \Z} \in \ell_{2, \rho}(\Z; H)$ we may set $F(u) \coloneqq (f(u_n))_{n \in \Z}$
    in Proposition \ref{p:equ-frac-de-seq-eq}.
\end{remark}

\begin{remark}[Gr\"unwald-Letnikov difference operator]
The Gr\"unwald-Letnikov difference operator is defined for $h > 0$ and $\alpha \in (0, 1)$ by (c.f.\ \cite[p.\ 708]{lubich1986}):
\begin{equation}\label{e:GL-op}
    \tilde{\Delta}_h^\alpha: V^{h \N} \rightarrow V^{h \N}, \qquad u \mapsto \left(t \mapsto \frac{1}{h^\alpha} \sum_{k = 0}^{t/h} (-1)^k \binom{\alpha}{k} u_{t - kh}\right),
\end{equation}
where $h \N = \setm{h n}{n \in \N}$.
It can be shown (cf.\ \cite[p.\ 708]{lubich1986}, \cite[p.\ 43]{podlubny1999}) that for $V = \R$ the Gr\"unwald-Letnikov operator
	can be used to approximate the Riemann-Liouville integral of sufficiently smooth functions.

Let $\alpha \in (0, 1)$.
For $v \in \ell_{2, \rho}(\N; H)$ and $u \coloneqq \iota v$ we calculate for the Gr\"unwald-Letnikov operator \eqref{e:GL-op}, $(1 - \tau^{-1})^\alpha u = \tilde{\Delta}_1^\alpha v$.
Let $h > 0$, $x \in H$ and $F: H \rightarrow H$.
A Gr\"unwald-Letnikov difference equation has the form
\begin{equation*}
	(\tilde{\Delta}_h^\alpha v)(t + h) = F(v(t)), \quad v(0) = x \qquad (t \in h \N).
\end{equation*}
For $h = 1$ the Grünwald-Letnikov equation resembles the Riemann-Liouville equation of Proposition \ref{p:equ-frac-de-seq-eq}
    and for $h \in \R_{> 0}$ we may treat a Grünwald-Letnikov problem by considering the problem
\begin{equation*}
    \tau (1 - \tau^{-1})^\alpha u = h^\alpha F(u) + \delta_{-1} x.
\end{equation*}
\end{remark}

\subsection*{Linear equations on sequence spaces}

\begin{remark}
Let $A \in L(H)$ and $x \in H$.
In view of the Riemann-Liouville difference operator we ask whether the linear equation 
\begin{equation}\label{e:lin-RL}
    \tau (1 - \tau^{-1})^\alpha u = A u + \delta_{-1} x
\end{equation}
of Proposition \ref{p:equ-frac-de-seq-eq} has a unique so-called causal solution that is supported in $\N$.
In the spaces $\ell_{2, \rho}(\Z; H)$ we have a unique solution of \eqref{e:lin-RL} for every initial value if $\tau (1 - \tau^{-1})^{\alpha} - A$ is invertible in $\ell_{2, \rho}(\Z; H)$.
In view of Proposition \ref{p:equ-frac-de-seq-eq} the solution $(\tau (1 - \tau^{-1})^\alpha - A)^{-1} \delta_{-1} x$ should be causal.
For the corresponding Caputo equation
\begin{equation}\label{e:lin-C}
    \tau (1 - \tau^{-1})^\alpha u = A u + (1 - \tau^{-1})^\alpha \chi_{\Z_{\geq -1}} x,
\end{equation}
the treatment is similar since $\chi_{\Z_{\geq -1}} x = \chi_{\N} x + \delta_{-1} x$.
\end{remark}

\begin{lemma}\label{lem:invertible}
Let $\alpha \in (0, 1)$ and $A \in L(H)$.
We define $f: \C_{\abs{\cdot} > 1} \rightarrow \C, z \mapsto z (1 - z^{-1})^\alpha$ and set $f_\rho \coloneqq f|_{S_\rho}$ for $\rho > 1$.
For $\rho > 1$ the operator $\tau (1 - \tau^{-1})^\alpha - A$ is invertible in $\ell_{2, \rho}(\Z; H)$ if and only if $\ran f_\rho \cap \sigma(A) = \emptyset$.
Moreover there is $\rho > 1$ such that for all $\mu > \rho$, $\ran f_\mu \cap \sigma(A) = \emptyset$,
	that is $\setm{z (1 - z^{-1})^\alpha}{\abs{z} > \rho}$ is in the resolvent set of $A$.
\end{lemma}

\begin{proof}
Recall the multiplication operator $\mulop$ of Lemma \ref{lem:mo}. 
Using the $\cZ$ transform, the operator $\tau (1 - \tau^{-1})^\alpha - A$ is invertible in $\ell_{2, \rho}(\Z; H)$ if and only if $\mulop (1 - \mulop^{-1})^\alpha - A$ is invertible in $L_2(S_\rho, H)$, since $\cZ_\rho$ is unitary.
This is the case, however, if and only if $\ran f_\rho \cap \sigma(A) = \emptyset$.
Using Lemma \ref{lem:binom-estimate} there is $\rho > 1$ such that for all
	$\mu > \rho$ and $z \in S_\mu$, $r(A) < \mu (1 - \mu^{-1})^\alpha \leq \abs{z (1 - z^{-1})^\alpha}$.
That is for all $\mu > \rho$, $\ran f_\mu \cap \sigma(A) = \emptyset$.
\end{proof}

\begin{proposition}[Causality of $(\tau (1 - \tau^{-1})^{\alpha} - A)^{-1}$]\label{p:causality-lin-frac}
Let $\rho > 1$, $\alpha \in (0, 1)$ and $A \in L(H)$.
Let $f_\rho$ be defined as in Lemma \ref{lem:invertible}.
The following are equivalent:
\begin{align*}
    &(i) &&(\tau (1 - \tau^{-1})^\alpha - A)^{-1} \in L\big(\ell_{2, \rho}(\Z; H)\big) \; \text{is causal},
    \\
    &(ii) &&(\tau (1 - \tau^{-1})^\alpha - A)^{-1} \in L\big(\ell_{2, \rho}(\Z; H)\big)
    \\& &&\text{and} \qquad \forall x \in H \colon \spt (\tau (1 - \tau^{-1})^{\alpha} - A)^{-1} \delta_{-1} x \subseteq \N,
    \\
    &(iii) &&\forall \mu \geq \rho \colon \ran f_\mu \cap \sigma(A) = \emptyset.
\end{align*}
\end{proposition}

\begin{proof}
$(i) \Rightarrow (ii)$:
Let $x \in H$ and $u \coloneqq (\tau (1 - \tau^{-1})^\alpha - A)^{-1} \delta_{-1} x$.
Using causality assumed in $(i)$, we obtain $\spt u \subseteq \Z_{\geq -1}$.
Moreover, $u_{-1} = ((1 - \tau^{-1})^{-\alpha} A u)_{-2} + ((1 - \tau^{-1})^{-\alpha} \delta_{-1} x)_{-2} = 0$ so that $\spt u \subseteq \N$.
\\[1ex]
$(ii) \Rightarrow (iii)$:
Suppose by contradiction that there is $\rho' > \rho$ with $\ran f_{\rho'} \cap \sigma(A) \neq \emptyset$.
The set $\setm{z \in \C_{\abs{\cdot} \geq \rho'}}{z(1 - z^{-1})^\alpha \in \sigma(A)}$ is closed, since $\sigma(A)$ is closed
	and since $f$ is continuous and the set is bounded, since by Lemma \ref{lem:invertible} there is a $\tilde \rho > \rho'$
	such that $f(\C_{\abs{\cdot} \geq \tilde \rho})$ is in the resolvent set.
Thus there is $z' \in \setm{z \in \C_{\abs{\cdot} \geq \rho'}}{z(1 - z^{-1})^\alpha \in \sigma(A)}$ with maximum absolute value.
Therefore there is a sequence $(z_n)_{n \in \N}$ in $\C$ with $\abs{z_n} > \abs{z'}$, that is $z_n (1 - z_n^{-1})^\alpha$ is in the resolvent set of $A$ ($n \in \N$)
	and $\lim_{n \rightarrow \infty} z_n = z'$.
Using the resolvent estimate (cf.\ \cite[p.~378]{werner2000}), we have $\lim_{n \rightarrow \infty} \norm{(z_n(1 - z_n^{-1})^\alpha - A)^{-1}}_{L(H)} = \infty$.
When applying the Banach-Steinhaus theorem (cf.\ \cite[p.~141]{werner2000}), there is $x \in H$ with $\lim_{n \rightarrow \infty} \norm{(z_n(1 - z_n^{-1})^\alpha - A)^{-1} x}_H = \infty$.
By assumption $(\tau (1 - \tau^{-1})^\alpha - A)^{-1} \delta_{-1} x \in \ell_{2, \rho}(\Z; H)$
	and $\spt (\tau (1 - \tau^{-1})^\alpha - A)^{-1} \delta_{-1} x \subseteq \N$.
Hence for $v \coloneqq (\tau (1 - \tau^{-1})^\alpha - A)^{-1} \delta_{0} x \in \ell_{2, \rho}(\Z; H)$
	we have $v \in \ell_{2, \rho}(\Z; H)$ and $\spt v \subseteq \N$.
Applying Lemma \ref{lem:positive-support}, it follows that $F: \C_{\abs{\cdot} > \rho} \rightarrow H, z \mapsto \sum_{k = -\infty}^\infty v_k z^{-k}$
	is analytic.
Since $v \in \ell_{2, \mu}(\Z; H)$ for $\mu > \abs{z'}$, it follows that for $G: \C_{\abs{\cdot} > \abs{z'}} \rightarrow H, z \mapsto (z(1 - z^{-1})^\alpha - A)^{-1} x$
	we have $G = F|_{\C_{\abs{\cdot} > \abs{z'}}}$.
This means that $\lim_{n \rightarrow \infty} \norm{F(z_n)}_H = \lim_{n \rightarrow \infty} \norm{G(z_n)}_H = \infty$.
Since $F$ is continuous, this is a contradiction in that $\lim_{n \rightarrow \infty} \norm{F(z_n)}_H\; \neq \infty$.
\\[1ex]
$(iii) \Rightarrow (i)$:
We have $(\tau (1 - \tau^{-1})^\alpha - A)^{-1} \in L(\ell_{2, \mu}(\Z; H))$ for $\mu > \rho$ by Lemma \ref{lem:invertible}.
Since the resolvent of $A$ is analytic, the mapping $z \mapsto (z (1 - z^{-1})^\alpha - A)^{-1}$ is analytic on $\C_{\abs{\cdot} > \rho}$.
Moreover the mapping $z \mapsto \norm{(z (1 - z^{-1})^\alpha - A)^{-1}}_{L(H)}$ is continuous and hence bounded on compact sets
	$\C_{\mu \geq \abs{\cdot} \geq \rho}$ where $\mu \geq \rho$, i.e. the mapping attains its maximum on $\C_{\mu \geq \abs{\cdot} \geq \rho}$.
By Lemma \ref{lem:binom-estimate} and since $A$ is bounded,
	$\sup_{z \in S_\mu} \norm{(z (1 - z^{-1})^\alpha - A)^{-1}}_{L(H)}$ decays to zero when $\mu$ tends to infinity.
It follows that $\mu \mapsto \sup_{z \in S_\mu} \norm{(z (1 - z^{-1})^\alpha - A)^{-1}}_{L(H)}$ is bounded on $[\rho, \infty)$ and therefore
	the conditions of Lemma \ref{lem:positive-support}$(ii)$ are satisfied for $(\tau (1 - \tau^{-1})^\alpha - A)^{-1} u$ where $u \in \ell_{2, \rho}(\Z; H)$, $\spt u \subseteq \N$.
It follows that $(\tau (1 - \tau^{-1})^\alpha - A)^{-1}$ is causal.
\end{proof}

\begin{remark}
Let $A \in L(H)$, $\rho > 1$ and $\alpha \in (0, 1)$.
By Lemma \ref{lem:invertible} and Proposition \ref{p:causality-lin-frac} we can always choose $\rho$ large enough such that $\tau (1 - \tau^{-1})^\alpha - A$ is invertible with causal inverse.
As a consequence the linear fractional difference equation \eqref{e:lin-RL} or \eqref{e:lin-C} has a unique solution $u \in \ell_{2, \rho}(\Z; H)$.
Moreover, from the previous Theorem it follows that \eqref{e:lin-RL} or \eqref{e:lin-C} has a unique solution in $\ell_{2, \mu}(\Z; H)$ for $\mu \geq \rho$
	which coincides with the solution $u$, since $\ell_{2, \rho}(\N; H) \subseteq \ell_{2, \mu}(\N; H)$.
Therefore we can speak of the solution operator $(\tau (1 - \tau^{-1})^\alpha - A)^{-1}$.

The difference equation for an initial value $x \in H$ and $A \in L(H)$
\begin{equation*}
    (\Delta^\alpha u)_n = A u_n, \qquad u_0 = x,
\end{equation*}
or
\begin{equation*}
    (\Delta_C^\alpha u)_n = A u_n, \qquad u_0 = x,
\end{equation*}
can be solved algebraically with a unique solution $u \in H^\N$ (cf.\ Proposition \ref{p:equ-frac-de-seq-eq}$(iii), (vi)$).
Recall the embedding $\iota$ of Proposition \ref{lem:one-sided}.
Since $A$ has bounded spectrum, when applying the previous theorem, there is $\rho > 1$ such that $\iota u \in \ell_{2, \rho}(\Z; H)$ is the unique solution of \eqref{e:lin-RL} or \eqref{e:lin-C}.
\end{remark}

\subsection*{Asymptotic stability}

We discuss asymptotic stability of linear fractional difference equations.
For an analysis of rates of convergence, see also \cite{czermak2015} and \cite{Tuan2018}.

\begin{definition}[Asymptotic stability]
Let $A \in L(H)$.
The zero equilibrium of equation \eqref{e:lin-RL} or \eqref{e:lin-C}, i.e., the solution $u = 0$ for the inital value $0$, is said to be asymptotically stable
    if for every $\rho > 1$, every solution $u \in \ell_{2, \rho}(\Z; H)$ of \eqref{e:lin-RL} or \eqref{e:lin-C} with $\spt u \subseteq \N$
    satisfies $\lim_{n \rightarrow \infty} u_n = 0$ in $H$.
\end{definition}

\begin{remark}\label{r:good-spaces}
If a sequence $u \in H^\Z$ satisfies $\spt u \subseteq \N$ and $\lim_{n \rightarrow \infty} u_n = 0$ then necessarily for all $\rho > 1$ we have $u \in \ell_{2, \rho}(\Z; H)$.
One could say that the spaces $\ell_{2, \rho}(\Z; H)$, $\rho > 1$, are large enough to look for asymptotically stable solutions of a linear sequence equation.
\end{remark}

\begin{proposition}[Necessary condition for asymptotic stability]\label{p:necessary}
Let $A \in L(H)$ such that the zero equilibrium of equation \eqref{e:lin-RL} or \eqref{e:lin-C} is asymptotically stable and let $f_\mu$ ($\mu > 1$) be as in Lemma \ref{lem:invertible}.
Then for all $\mu > 1$, $\tau(1 - \tau^{-1})^\alpha - A$ is invertible in $\ell_{2, \mu}(\Z; H)$ with causal inverse, i.e., for each $\mu > 1$, $\sigma(A) \cap \ran f_\mu = \emptyset$.
\end{proposition}

\begin{proof}
Assume by contradiction there is $z' \in \ran f_\rho \cap \sigma(A) \neq \emptyset$ where $\rho > 1$.
We may assume that $\ran f_\mu \cap \sigma(A) = \emptyset$ for $\mu > \abs{z'}$.
Then there is a sequence $(z_n)_{n \in \N}$ with $\abs{z_n} > \abs{z'}$ such that $z_n (1 - z_n^{-1})^\alpha$ is in the resolvent set of $A$ ($n \in \N$) and such that $z_n \rightarrow z'$ ($n \rightarrow \infty$).
Using the resolvent estimate we have $\lim_{n \rightarrow \infty} \norm{(z_n (1 - z_n^{-1})^\alpha - A)^{-1}}_{L(H)} = \infty$.
Using the Banach-Steinhaus theorem there is $x \in H$ with $\lim_{n \rightarrow \infty} \norm{(z_n (1 - z_n^{-1})^\alpha - A)^{-1} x}_H = \infty$.
By Lemma \ref{lem:invertible} and Proposition \ref{p:causality-lin-frac}, for $\mu > \abs{z'}$ we know that $\tau (1 - \tau^{-1})^\alpha - A$ is invertible in $\ell_{2, \mu}(\Z; H)$
	and $v \coloneqq (\tau (1 - \tau^{-1})^\alpha - A)^{-1} \delta_0 x$ satisfies $\spt v \subseteq \N$.
Since the zero equilibrium is asymptotically stable, we have $v \in \ell_{2, \rho'}(\Z; H)$ for some $\rho' \in (1, \abs{z'})$ by Remark \ref{r:good-spaces}.
Then the mapping $F: \C_{\abs{\cdot} > \rho'} \rightarrow H, z \mapsto \sum_{k = -\infty}^\infty v_k z^{-k}$ is analytic and
	equals $G: \C_{\abs{\cdot} > \abs{z'}} \rightarrow H, z \mapsto (z(1 - z^{-1})^\alpha - A)^{-1} \delta_0 x$ on $\C_{\abs{\cdot} > \abs{z'}}$
	by Lemma \ref{lem:positive-support}.
Therefore we have $\lim_{n \rightarrow \infty} F(z_n) < \infty$, since $F$ is analytic which contradicts $\lim_{n \rightarrow \infty} F(z_n) = \lim_{n \rightarrow \infty} G(z_n) = \infty$. 
\end{proof}

For a sufficient condition of asymptotic stability we observe that if $u \in \ell_{2, 1}(\Z; H)$ with $\spt u \subseteq \N$ then $\lim_{n \rightarrow \infty} u_n = 0$.

\begin{proposition}[Sufficient condition for asymptotic stability]\label{p:sufficient}
Let $A \in L(H)$.
For all $\rho > 1$ let $\tau(1 - \tau^{-1})^\alpha - A$ be invertible in $\ell_{2, \rho}(\Z; H)$ with causal inverse.
If for all $x \in H$ the mapping $\C_{\abs{\cdot} > 1} \rightarrow H, z \mapsto \sum_{k = -\infty}^\infty [(\tau(1 - \tau^{-1})^\alpha - A)^{-1} \delta_{-1} x]_k z^{-k}$
	has a continuous continuation to the unit circle $S_1$ then the zero equilibrium of equation \eqref{e:lin-RL} or \eqref{e:lin-C} is asymptotically stable.
\end{proposition}

\begin{proof}
Let $g$ be the continuous continuation.
Then $g|_{S_1} \in L_2(S_1, H)$ and $v \coloneqq \cZ_1^{-1} g|_{S_1} \in \ell_{2, 1}(\Z; H)$.
Moreover, $u = v$ that is $u \in \ell_{2, 1}(\Z; H)$.
\end{proof}

\begin{remark}
We believe that the necessary conditions for stability in Proposition \ref{p:necessary} are not sufficient, neither are the sufficient conditions for stability in Proposition \ref{p:sufficient} necessary. Already for semigroups the asymptotic stability can in general not be characterized by spectral conditions solely. The shift operator on continuous functions from $\R^+$ to $\R$ which decay at infinity, for example, is asymptotically stable although its spectrum consists of all complex numbers with non-positive real part \cite[Example 2.5(c)]{Arendt:Batty1988}. The characterization of asymptotic stability for linear fractional difference equations is an intricate problem which still needs to be addressed.
\end{remark}

\begin{example}
Let $H = \C$, $A: \C \rightarrow \C, z \mapsto \lambda z$ where $\lambda \in \R$ and $\alpha \in (0, 1)$.
We study the asymptotic behavior of the linear fractional equations \eqref{e:lin-RL} and \eqref{e:lin-C} on $\ell_{2, \rho}(\Z; H)$ ($\rho > 1$) in view of Proposition \ref{p:necessary} and Proposition \ref{p:sufficient}
    and therefore want to apply the $\cZ$ transform to equation \eqref{e:lin-RL} and \eqref{e:lin-C}.
In order to obtain an asymptotically stable zero equilibrium by Proposition \ref{p:necessary}, we must have $\sigma(A) \cap \ran f = \emptyset$ where $f: \C_{\abs{\cdot} > 1} \rightarrow \C, z \mapsto z (1 - z^{-1})^\alpha$
    is defined as in Lemma \ref{lem:invertible} and $\sigma(A) = \{\lambda\}$.
We remark that for $z \in \C_{\abs{\cdot} > 1}$, $f(z) \in \R$ if and only if $z \in \R$ since $f$ is injective and since $f(\overline z) = \overline{f(z)}$.
Moreover $f(\C_{\abs{\cdot} > 1} \cap \R) = (-\infty, -2^\alpha) \cup (0, \infty)$ and so $\lambda \notin \ran f$ if and only if $\lambda \in [-2^\alpha, 0]$.
By Proposition \ref{p:necessary} we necessarily have $\lambda \in [-2^\alpha, 0]$ if the zero equilibrium of \eqref{e:lin-RL} or \eqref{e:lin-C} is asymptotically stable.
Let $\lambda \in [-2^\alpha, 0]$ and for $u \in \ell_{2, \rho}(\Z; H)$ we denote $\hat{u} \coloneqq \cZ u$.

We consider \eqref{e:lin-RL} with $x \in \C$ first.
Also for $z \in S_\rho$ we have $(\cZ \delta_{-1} x)(z) = z x$.
Applying the $\cZ$ transform to equation \eqref{e:lin-RL}, we obtain for $z \in S_\rho$
\begin{equation*}
    z (1 - z^{-1})^\alpha \hat{u}(z) = A \hat{u}(z) + z x.
\end{equation*}
If $\lambda \in (-2^\alpha, 0)$ the mapping $\C_{\abs{\cdot} > 1} \rightarrow H, z \mapsto \frac{z x}{z (1 - z^{-1})^\alpha - \lambda}$ has a continuous continuation to $S_1$
    and by Proposition \ref{p:sufficient} we obtain that the zero equilibrium of \eqref{e:lin-RL} is asymptotically stable.

We now consider equation \eqref{e:lin-C} where $x \in \C$.
For $z \in S_\rho$ we have $(\cZ \chi_{\Z_{\geq 1}} x)(z) = \frac{z x}{1 - z^{-1}}$.
Applying the $\cZ$ transform to equation \eqref{e:lin-C}, we obtain for $z \in S_\rho$
\begin{equation*}
    z (1 - z^{-1})^\alpha \hat{u}(z) = A \hat{u}(z) + z (1 - z^{-1})^{\alpha - 1} x.
\end{equation*}
If $\lambda \in (-2^\alpha, 0)$ the mapping $\C_{\abs{\cdot} > 1} \rightarrow H, z \mapsto \frac{z (1 - z^{-1})^{\alpha - 1} x}{z (1 - z^{-1})^\alpha - \lambda}$ has a continuous continuation to $S_1$
    and using Proposition \ref{p:sufficient} we obtain that the zero equilibrium of \eqref{e:lin-C} is asymptotically stable.
    
The cases $\lambda = 0$ and $\lambda = -2^\alpha$ are discussed in \cite{czermak2015}.
\end{example}

\section*{Acknowledgement}

The research of A.B.\ and A.C.\ was funded by the National Science Centre in Poland granted according to decisions DEC-2015/19/D/ST7/03679 and DEC-2017/25/B/ST7/02888, respectively.
The research of M.N.\ was supported by the Polish National Agency for Academic Exchange according to the decision PPN/BEK/2018/1/00312/DEC/1.
The research of S.S.\ was partially supported by an Alexander von Humboldt Polish Honorary Research Fellowship. 
The work of H.T.\ Tuan was supported by the joint research project from RAS and VAST QTRU03.02/18-19.  

\bibliography{fracdifference}{}
\bibliographystyle{abbrv}

\end{document}